\documentclass[reqno,11pt]{amsart}

\usepackage{mathrsfs}
\usepackage{amsmath,amssymb,slashed}
\usepackage{color,xcolor}
\usepackage[english]{babel}
\usepackage[utf8]{inputenc}
\usepackage{amsmath,amscd}
\usepackage{amsfonts}
\usepackage{amsthm}
\usepackage{graphicx}
\usepackage[colorinlistoftodos]{todonotes}
\usepackage{amsmath,amssymb,amsfonts}
\usepackage{underscore}
\usepackage{xy}
\usepackage[T1]{fontenc}
\usepackage[utf8]{inputenc}
\usepackage{amsthm}
\usepackage{mathrsfs}
\usepackage{amsmath,amssymb,slashed}
\usepackage{color,xcolor}

\theoremstyle{plain}
\newtheorem{thm}{Theorem}[section]

\newtheorem{lem}[thm]{Lemma}
\newtheorem{cor}[thm]{Corollary}
\newtheorem{eg}[thm]{Example}

\theoremstyle{definition}
\newtheorem{defn}[thm]{Definition}

\theoremstyle{remark}
\newtheorem{rmk}[thm]{Remark}

\newcommand{\bC}{\mathbb{C}}

\newcommand{\bR}{\mathbb{R}}

\newcommand{\la}{\langle}
\newcommand{\ra}{\rangle}

\begin{document}

\title[$p$-nuclearity of $L^p$-operator crossed products]{$p$-nuclearity of $L^p$-operator crossed products}

\author[Z. Wang]{Zhen Wang}
\curraddr{Department of Mathematics\\Hangzhou Normal University\\Hangzhou 311121\\P.~R. China}
\address{Department of Mathematics\\Jilin University\\Changchun 130012\\P.~R. China}
%\address{Key Laboratory of Financial Mathematics of Fujian Province University\\School of Mathematics and Finance \\Putian University \\Putian 351100\\
%P. R. China}
\email{wangzhen@hznu.edu.cn}

\author[S. Zhu]{Sen Zhu}
\address{Department of Mathematics\\Jilin University\\Changchun 130012\\P.~R. China}
\email{zhusen@jlu.edu.cn}
%\thanks{This work was supported by NNSFC (grant numbers 12201240, 12171195) and China Postdoctoral Science Foundation (grant number 2022M711310)}


\begin{abstract}
Let $(X,\mathcal{B},\mu)$ be a measure space and $A$ be a norm closed subalgebra of $\mathcal{B}(L^p(X,\mu))$, where $p\in [1,\infty)$.
Let $(G,A,\alpha)$ be an $L^p$-operator algebra dynamical system, where $G$ is a countable discrete amenable group. We prove that the full $L^p$-operator crossed product $F^p(G,A,\alpha)$ is $p$-nuclear if and only if $A$ is $p$-nuclear {provided the action} $\alpha$ of $G$ on $A$ is $p$-completely isometric. As applications, we prove that $L^p$-Cuntz algebras and rotation $L^p$-operator algebras are $p$-nuclear. Our results solve { a problem raised by N. C. Phillips concerning {$p$-nuclearity} for $L^p$-Cuntz algebras.}
\end{abstract}

\date{\today}
\maketitle

 %\tableofcontents

\section{Introduction}

The aim of this paper is to study { $p$-nuclearity} for $L^p$-operator crossed products. To proceed, we give a brief introduction to $L^p$-operator algebras.

\subsection{$L^p$-operator algebras}
For $p\in[1,\infty)$, we say that a Banach algebra $A$ is an {\it $L^{p}$-operator algebra} if
it is isometrically isomorphic to a norm closed subalgebra of the algebra $\mathcal{B}(E)$ of all bounded operators on some $L^p$-space $E$.
Clearly, $L^p$-operator algebras {are} a natural generalization of
operator algebras on Hilbert spaces (and in particular $C^*$-algebras) by replacing Hilbert spaces with $L^p$-spaces.

The study of $L^p$-operator algebras can be traced back to C. Herz's influential works \cite{Herz,Herz1,Herz2} on harmonic analysis of group algebras on $L^p$-spaces in 1970s.
Given a locally compact group $G$, C. Herz \cite{Herz} introduced the Banach algebra $F^p_\lambda(G)$ (originally denoted by $PF_p(G)$),
constructed from the left regular representation of $G$ on $L^p(G)$, which is an $L^p$-analogue of the reduced group $C^*$-algebra $C^*_\lambda(G)$. Associated with $F^p_\lambda(G)$ there are two other natural algebras, { namely} the weak-$*$ completion $PM_p(G)$ of $F^p_\lambda(G)$ and the double commutant $CV_p(G)$ of $F^p_\lambda(G)$ {in $\mathcal{B}(L^p(G))$}. The reader is referred to \cite{Daws,Derighetti} for more research
concerning them, especially on the problem of determining whether $PM_p(G)=CV_p(G)$.

The last decade witnessed a revived interest in $L^p$-operator algebras, mainly inspired by the works of
N. C. Phillips \cite{Phillips Lp Cuntz,Odp,Lp UHF,N. C. Phillips Lp} which provide new ideas, examples and techniques from operator algebras.
In particular, he introduced $L^{p}$-Cuntz algebras, and showed that these Banach algebras
behave in many ways very similarly to the corresponding $C^*$-algebras.
In order to compute the $K$-theory groups of $L^{p}$-Cuntz algebras, N. C. Phillips \cite{N. C. Phillips Lp} introduced $L^p$-operator crossed products and proved that the $K$-theory groups of $L^{p}$-Cuntz algebras are independent of the choice of $p\in [1,\infty)$.

The work of N. C. Phillips has motivated other authors to introduce $L^p$-analogs of classical operator algebras on Hilbert spaces
and study whether one can establish $L^p$-analogues of important results from the theory of $C^*$-algebras. Along this line, group algebras \cite{Gardella and Thiel,Gardella and Thiel iso}, crossed products \cite{Gardella and Thiel convolution,Hejazian and Pooya,WZ}, groupoid algebras \cite{Choi and,Gardella and Lupini groupoid}, AF algebras \cite{Gardella and Lupini groupoid UHF,Lp AF}, graph algebras \cite{Cort}, the $L^p$-Roe type algebras \cite{Braga,Chung lp Roe 1,Chung lp Roe,Li Kang,shan,S and Z,zhang} and the $l^p$-Toeplitz algebra \cite{WW} are studied. Although most previous investigations have been very largely focused on various examples,
some recent works were undertaken in a more abstract and systematic way \cite{Blecher and Phillips,Choi and,Gardella and Thiel isometry,Gardella and Thiel bidual}, indicating that there is a rich general theory waiting to be discovered.

Recently there has been growing interest in $L^p$-operator algebras.
The influx of operator-algebraic techniques leads to progresses on some long standing questions. For example, it was shown in \cite{Gardella and Thiel quotient} that the class of $L^p$-operator algebras ($p\ne 2$) is not closed under quotients, answering an old open question asked by Le Merdy \cite{Merdy}.
In \cite{Chung lp Roe 1}, Y. C. Chung and K. Li showed that coarse equivalence is implied by an isometric isomorphism of $l^p$-Roe-type algebras without the assumption of Yu's property A. This exhibits an interesting interplay between coarse geometry and $L^p$-operator algebras.
The reader is referred to \cite{Gardella modern} for more historical comments and recent developments in $L^p$-operator algebras.

$L^p$-operator crossed products were introduced by N. C. Phillips \cite{N. C. Phillips Lp} with the aim to compute the $K$-theory groups of $L^{p}$-Cuntz algebras. The stabilized $L^{p}$-Cuntz algebras can be realized as the reduced $L^p$-operator crossed products by $\mathbb{Z}$ (see \cite[Theorem 7.17]{N. C. Phillips Lp}).
Next we recall some notations and terminology.

\subsection{$L^p$-operator crossed products}\label{cross}
Throughout the following, discrete groups will always be implicitly endowed with the counting measure.
We say that an isometric representation $\pi$ of an $L^p$-operator algebra $A$ on $\mathcal{B}(L^p(X,\mu))$ is {\it $\sigma$-finite} if $\mu$ is $\sigma$-finite, and $\pi$ is {\it separable} if $L^p(X,\mu)$ is separable.

An $L^{p}$-{\it operator algebra dynamical system} is a triple $(G,A,\alpha)$ consisting of a discrete group $G$, an $L^{p}$-operator algebra $A$ and a continuous homomorphism $\alpha:G\rightarrow \mathrm{Aut}(A)$, where $\mathrm{Aut}(A)$ is the group of isometric automorphisms of $A$.
A {\it contractive covariant representation} of $(G,A,\alpha)$ on an $L^{p}$-space $E$ is a pair $(\pi,v)$ consisting of a nondegenerate contractive homomorphism $\pi:A\rightarrow \mathcal{B}(E)$ and an isometric group representation $v:G\rightarrow \mathcal{B}(E)$ satisfying the covariance condition $$v_{t}\pi(a)v_{t^{-1}}=\pi(\alpha_{t}(a)), \ \ t\in G, a\in A.$$
A covariant representation $(\pi,v)$ of $(G,A,\alpha)$ is said to be {\it $\sigma$-finite} if $\pi$ is $\sigma$-finite.

We denote by $C_{c}(G,A,\alpha)$ the vector space of compactly supported continuous functions from $G$ to $A$, made into an algebra over $\mathbb{C}$ with product given by twisted convolution, that is, $$(f*g)(t):=\sum_{s\in G}f(s)\alpha_{s}\left(g(s^{-1}t)\right)$$ for $f,g\in C_{c}(G,A,\alpha)$ and $t\in G$. The {\it integrated form} of $(\pi,v)$ is the nondegenerate contractive homomorphism $\pi\rtimes v: C_{c}(G,A,\alpha)\rightarrow \mathcal{B}(E)$ given by $$(\pi\rtimes v)(f)(\xi):=\sum_{t\in G}\pi\left(f(t)\right)v_{t}(\xi)$$ for $f\in C_{c}(G,A,\alpha)$ and $\xi\in E$.

We let $\mathrm{Rep}_{p}(G,A,\alpha)$ denote the class of all nondegenerate $\sigma$-finite contractive covariant representations of $(G,A,\alpha)$ on $L^{p}$-spaces. The {\it full $L^p$-operator crossed product} $F^{p}(G,A,\alpha)$ is defined as the completion of $C_{c}(G,A,\alpha)$ in the norm $$||f||_{F^{p}(G,A,\alpha)}:=\sup\big\{||(\pi\rtimes v)(f)||:(\pi,v)\in \mathrm{Rep}_{p}(G,A,\alpha)\big\}.$$

Given two measure spaces $(X,\mu)$ and $(Y,\nu)$, we denote by $L^p(X,\mu)\otimes_p L^p(Y,\nu)$ the $L^p$-tensor product, which can be canonical identified with $L^p(X\times Y,\mu\times\nu)$ via $\xi\otimes\eta(x,y)=\xi(x)\eta(y)$ for $\xi\in L^p(X,\mu)$ and $\eta\in L^p(Y,\nu)$.

Let $(G,A,\alpha)$ be an $L^{p}$-operator algebra dynamical system.
Given a nondegenerate $\sigma$-finite contractive representation $\pi_{0}:A\rightarrow \mathcal{B}(E_{0})$ on an $L^{p}$-space $E_{0}$, its associated {\it regular covariant representation} is the pair $(\pi,\lambda_{p}^{E_{0}})$ on $l^{p}(G)\otimes_{p} E_{0}\cong l^{p}(G,E_{0})$ given by
$$\pi(a)(\xi)(s):=\pi_{0}\left(\alpha_{s^{-1}}(a)\right)(\xi(s)) \quad$$ and $$\lambda_p^{E_0}(s)(\xi)(t):=\xi(s^{-1}t)$$
for $a\in A$, $\xi\in l^{p}(G,E_{0})$ and $s,t\in G$.
We denote by $\mathrm{RegRep}_{p}(G,A,\alpha)$ the class consisting
of nondegenerate $\sigma$-finite contractive regular covariant representations of $(G,A,\alpha)$, which is clearly a subclass of $\mathrm{Rep}_{p}(G,A,\alpha)$. The {\it reduced $L^p$-operator crossed product} $F^{p}_{\lambda}(G,A,\alpha)$ is defined as the completion of $C_{c}(G,A,\alpha)$ in the norm $$||f||_{F^{p}_{\lambda}(G,A,\alpha)}:=\sup\{||(\pi\rtimes v)(f)||:(\pi,v)\in \mathrm{RegRep}_{p}(G,A,\alpha)\}.$$

By \cite[Theorem 7.1]{Phillips look like}, if $G$ is amenable, then  $F^{p}(G,A,\alpha)$ is isometrically isomorphic to $F^{p}_{\lambda}(G,A,\alpha)$. If $A=\mathbb{C}$, then it is easy to see that $F^{p}(G,A,\mathrm{id})$ is the full group $L^{p}$-operator algebra $F^{p}(G)$ and $F^{p}_{\lambda}(G,A,\mathrm{id})$ is the reduced group $L^{p}$-operator algebra $F^{p}_{\lambda}(G)$,
where $\mathrm{id}$ is the trivial action of $G$ on $\mathbb{C}$. When $A=C(X)$ for some compact Hausdorff space $X$, we simply write $F^p(G,X,\alpha)$ for $F^p(G,C(X),\alpha)$ and write $F^p_\lambda(G,X,\alpha)$ for $F^p_\lambda(G,C(X),\alpha)$.

\subsection{$p$-nuclearity}
The aim of the present paper is to study $p$-nuclearity of $L^p$-operator crossed products.
We first recall the nuclear property of $C^*$-algebras, which was introduced by Takesaki \cite{Takesaki}.

A $C^*$-algebra $A$ is said to be {\it nuclear} if for any $C^*$-algebra $B$ there is a unique norm on the algebraic tensor product $A\odot B$.
By influential works of Lance \cite{Lance}, Choi-Effros \cite{Choi and Effros} and Kirchberg \cite{Kirchberg}, the nuclearity is equivalent to {\it completely positive approximation property}, that is, for any finite subset $F$ of $A$ and $\varepsilon>0$, there exist a positive integer $n$ and two completely contractive positive maps $\varphi:A\rightarrow M_n$ and $\psi:M_n\rightarrow A$ such that $\|\psi\circ\varphi(a)-a\|<\varepsilon$ for all $a\in F$. Also it is well known that the nuclearity is equivalent to the amenability for $C^*$-algebras \cite{Connes,Haagerup}, which was originally introduced by B. E. Johnson \cite{Johnsom} for Banach algebras.

In \cite{an}, G. An, J.-J. Lee and Z.-J. Ruan introduced and studied $p$-nuclearity for reduced group $L^p$-operator algebras $F^p_\lambda(G)$.
To formulate their result, we first recall some notions.

For each positive integer $n$, {denote $l_n^p=L^p(\{1,2,\cdots,n\},\nu)$, where $\nu$ is the counting measure on $\{1,2,\cdots,n\}$.
We denote $M_n^p=\mathcal{B}(l_n^p)$}.
Given a norm closed subalgebra $A$ of $\mathcal{B}(L^p(X,\mu))$, we denote by $M_n^p\otimes_p A$ the {\it $L^p$-matrix algebra}, that is, the Banach subalgebra of $\mathcal{B}(L^p(\{1,2,\cdots,n\}\times X,\nu\times\mu))$ generated by all $T\otimes a$'s for $T\in M_n^p$ and $a\in \mathcal{B}(L^p(X,\mu))$. Clearly, each element of $M_n^p\otimes_p A$ is of form $[a_{i,j}]_{1\leq i,j\leq n}$ with $a_{i,j}\in A$,
which is also written as $\sum_{i,j=1}^n e_{i,j}\otimes a_{i,j}$, where $\{e_{i,j}\}_{1\leq i,j\leq n}$ are the canonical matrix units of $M_n^p$.

\begin{defn}
Let $A$ be a norm closed subalgebra of $\mathcal{B}(L^p(X,\mu))$, $B$ be a norm closed subalgebra of $\mathcal{B}(L^p(Y,\nu))$
and $\varphi$ be a linear map $\varphi: A\rightarrow B$.  We denote by $\mathrm{id}_{M_n^p}\otimes \varphi$ the map from $M_n^p\otimes_p A$ to $M_n^p\otimes_p B$ defined by $$\mathrm{id}_{M_n^p}\otimes \varphi\left(\sum_{i,j=1}^n e_{i,j}\otimes a_{i,j}\right)=\sum_{i,j=1}^n e_{i,j}\otimes \varphi(a_{i,j}) $$ for
$\sum_{i,j=1}^n e_{i,j}\otimes a_{i,j}\in M_n^p\otimes_p A$.
We denote $\|\varphi\|_{cb}=\sup_{n\in \mathbb{Z}_{>0}}\|\mathrm{id}_{M_n^p}\otimes \varphi\|$.
We say that $\varphi$ is {\it $p$-completely contractive} if $\|\varphi\|_{cb}\leq 1$, and say that $\varphi$ is {\it $p$-completely isometric} if $\mathrm{id}_{M_n^p}\otimes \varphi$ is isometric for all positive integer $n$.
\end{defn}

\begin{defn}[{\cite[Proposition 5.1 (a)]{an}}]
Let $(X,\mathcal{B},\mu)$ be a measure space and $A\subset \mathcal{B}(L^p(X,\mu))$ be a norm closed subalgebra.
We say that $A$ is $p$-nuclear if, for every finite subset $F$ of $A$ and every $\varepsilon>0$, there exist a positive integer $n$ and two $p$-completely contractive maps $\varphi:A\rightarrow M_n^p$ and $\psi:M_n^p\rightarrow A$ such that $\|\psi\circ \varphi(a)-a\|<\varepsilon$ for all $a\in F$.
\end{defn}

Given a discrete amenable group $G$, it was proved in \cite{an} that $F^p_\lambda(G)$ is $p$-nuclear.
When $A$ is a $C^*$-algebra, R. R. Smith proved that $2$-nuclearity is equivalent to nuclearity (see \cite[Theorem 1.1]{Smith}).
Moreover, if $G$ is a discrete amenable group and $\alpha:G\rightarrow \mathrm{Aut}(A)$ is an isometric action of $G$ on {a $C^*$-algebra} $A$, then, by
\cite[Theorem 1]{Rosenberg} and \cite[Theorem 4.2.6]{Brown and Ozawa}), the full crossed product $A\rtimes_\alpha G$ is nuclear if and only if $A$ is nuclear. In \cite{PlpsOpenQues}, N. C. Phillips proposed to study whether there is an $L^p$-analogue of the preceding result and raised several problems.

The main result of this paper is the following theorem, which determines when
the full $L^p$-operator crossed product $F^p(G, A, \alpha)$ is $p$-nuclear
in the case that $G$ is a countable discrete amenable group and $\alpha$ is a $p$-completely isometric action. Recall that an action $\alpha:G\rightarrow \mathrm{Aut}(A)$ is said to be {\it $p$-completely isometric} if $\alpha_t:A\rightarrow A$ is $p$-completely isometric for every $t\in G$.

\begin{thm}\label{T:main}
Let $p\in [1,\infty)$ and $(X,\mathcal{B},\mu)$ be a $\sigma$-finite measure space such that $E=L^p(X,\mu)$ is separable.
Let $A\subset\mathcal{B}(E)$ be a norm closed subalgebra such that the identity representation of $A$ on $E$ is nondegenerate, $G$ be a countable discrete amenable group and $\alpha:G\rightarrow \mathrm{Aut}(A)$ be a $p$-completely isometric action of $G$ on $A$. Then $F^p(G,A,\alpha)$ is $p$-nuclear if and only if $A$ is $p$-nuclear.
\end{thm}

The preceding result partially solves \cite[Problem 10.4]{PlpsOpenQues} raised by N. C. Phillips.

%Let $G$ be a non-amenable discrete group. Then $F^1(G)=F^1_\lambda(G)=l^1(G)$ is $1$-nuclear but not amenable (see \cite[Theorem 6.4]{an} and \cite[Theorem 2.1.10]{Runde}).

\begin{rmk}\label{completely isometric}
\begin{enumerate}
\item[(i)] We do not know whether the result of Theorem \ref{T:main} still holds when the condition on $\alpha$ is relaxed to an isometric action.
As we know, it is still unclear whether an isometric action of $G$ is necessarily $p$-completely isometric.
\item[(ii)] Clearly, if $A$ has unique $L^p$-operator matrix norms (see Definition \ref{matrix norm}), then each isometric action $\alpha$ of $A$ is $p$-completely isometric.
Both the stabilized spatial $L^p$-UHF algebra and $C(S^1)$ have unique $L^p$-operator matrix norms (see \cite[Corollary 4.4 \& Propositions 4.6 and 6.4]{Lp AF}), where $S^1$ denotes the unit circle in $\bC$.
\item[(iii)] { The amenability of an $L^p$-operator algebra does not coincide with the $p$-nuclearity. In fact, by \cite[Theorem 6.4]{an}, \cite[Proposition 2.11, Remark 2.12]{Gardella and Thiel} and \cite[Theorem 2.1.10]{Runde}, if $G$ is a non-amenable discrete group, then $F^1(G)$ and $F^1_\lambda(G)$ are $1$-nuclear but not amenable.
    For an amenable $L^p$-operator algebra which is not $p$-nuclear, see Example \ref{Eg:amena}. Hence the result of Theorem \ref{T:main} does not just follow from the amenability for $L^p$-operator crossed products (see \cite[Theorem 5.2]{N. C. Phillips Lp}).}
\end{enumerate}
\end{rmk}

%Next we shall give two applications of Theorem \ref{T:main}.
%First, we have the following result, solving the $p$-nuclearity problem for $L^p$-Cuntz algebras.

Next we shall give two applications of Theorem \ref{T:main}.

In \cite[Corollary 5.18]{Odp}, N. C. Phillips proved that $L^p$-Cuntz algebras  $\mathcal{O}_d^p$ are all amenable
(see Definition \ref{D:Cunz} for the precise definition). N. C. Phillips \cite[Theorem 7.17]{N. C. Phillips Lp} proved that the $L^p$-Cuntz algebra $\mathcal{O}_d^p$ is stably isomorphic to the reduced $L^p$-operator crossed product of a stabilized spatial $L^p$-UHF algebra by $\mathbb{Z}$; moreover, he pointed out that the spatial $L^p$-UHF algebra is $p$-nuclear \cite[Theorem 4.6]{Phillips look like}.

Using Theorem \ref{T:main}, we shall prove in Section 3 the following result.

\begin{cor}\label{C:3}
The $L^p$-Cuntz algebras are $p$-nuclear for $p\in [1,\infty)$.
\end{cor}

This answers \cite[Problem 10.5]{PlpsOpenQues} positively.

{

Finally, we consider $p$-nuclearity of rotation $L^p$-operator algebras $A_\theta^p$ (the precise definition will be provided in Definition \ref{D:rota}), which are essentially $L^p$-operator crossed products. In fact, E. Gardella and H. Thiel \cite{GarThiRotation14} proved that $A_\theta^p$ is isometrically isomorphic to $F^{p}(\mathbb{Z},F^p(\mathbb{Z}), \beta_\theta)$ for some
action $\beta_\theta$ of $\mathbb{Z}$ on $F^p(\mathbb{Z})$ (see the discussion after Definition \ref{D:rota} for more details).

Next we introduce another $L^p$-operator crossed product closely related to $A_\theta^p$.
For $\theta\in \mathbb{R}$, we denote by $\alpha_\theta$ the isometric automorphism of $C(S^1)$ given by  $\alpha_{\theta}(f)(z)=f(e^{-2\pi \textrm{i}\theta} z)$ for $z\in S^1$.
The action of $\mathbb{Z}$ on $C(S^1)$ given by $n\longmapsto \alpha_\theta^n$ is still denoted by $\alpha_{\theta}$. Let $\Gamma: F^p(\mathbb{Z})\rightarrow C(S^1)$ be the Gelfand transform of $F^p(\mathbb{Z})$. The reader will see in Section 3 that $\Gamma$ induces a canonical injective homomorphism from $F^p(\mathbb{Z},F^p(\mathbb{Z}),\beta_\theta)$ to $ F^p(\mathbb{Z},S^1,\alpha_\theta)$, which is surjective precisely when $p=2$.

Applying Theorem \ref{T:main}, we shall prove the following.

{\begin{cor}\label{2}
Let $\theta\in \mathbb{R}$ and $p\in [1,\infty)$. Then both
$F^{p}(\mathbb{Z},S^1,\alpha_\theta)$ and $A_\theta^p$ are $p$-nuclear.
\end{cor}}

The rest of this paper is organized as follows.
In Section 2, we shall provide some auxiliary results and give the proof of Theorem \ref{T:main}. In Section 3, we shall give the proofs of Corollaries \ref{C:3} and \ref{2}.

}

\section{Proof of the main result}

Before proving Theorem \ref{T:main}, we first prove several useful lemmas.

\subsection{Auxiliary results}

{\begin{lem}\label{matrix}
%Let $p\in [1,\infty)$ and $(X,\mu), (Y,\nu)$ be two measure spaces. Let $A\subset \mathcal{B}(L^p(X,\mu))$ and $B\subset \mathcal{B}(L^p(Y,\nu))$ be two norm closed subalgebras with $B$ being $p$-nuclear.
%If, for every finite subset $F$ of $A$ and every $\varepsilon>0$, there exist two $p$-completely contractive maps $\varphi:A\rightarrow B$ and $\psi:B\rightarrow A$ such that $\|\psi\circ \varphi(a)-a\|<\varepsilon$ for all $a\in F$, then $A$ is $p$-nuclear.

Let $p\in [1,\infty)$, $(X,\mu)$ be a measure space and $A\subset \mathcal{B}(L^p(X,\mu))$ be a norm closed subalgebra. If, for every finite subset $F$ of $A$ and every $\varepsilon>0$, there exist a $p$-nuclear $L^p$-operator algebra $B\subset \mathcal{B}(L^p(Y,\nu))$ with $ (Y,\nu)$ being a measure space, and two $p$-completely contractive maps $\varphi:A\rightarrow B$ and $\psi:B\rightarrow A$ such that $\|\psi\circ \varphi(a)-a\|<\varepsilon$ for all $a\in F$, then $A$ is $p$-nuclear.
\end{lem}

\begin{proof}
Arbitrarily choose a finite subset $F$ of $A$ and an $\varepsilon>0$. By the hypothesis, there exist a $p$-nuclear $L^p$-operator algebra $B$ and two $p$-completely contractive maps $\varphi:A\rightarrow B$ and $\psi:B\rightarrow A$ such that $$\|\psi\circ \varphi(a)-a\|<\frac{\varepsilon}{2}$$ for all $a\in F$. Clearly, $\varphi(F)$ is a finite subset of $B$. Since $B$ is $p$-nuclear, for $\varepsilon>0$ and the finite subset $\varphi(F)$ of $B$, there exist a positive integer $n$ and two $p$-completely contractive maps $\varphi_0:B\rightarrow M_n^p$ and $\psi_0:M_n^p\rightarrow B$ such that $$\|\psi_0\circ\varphi_0(b)-b\|<\frac{\varepsilon}{2}$$ for all $b\in \varphi(F)$. Then $$\|(\psi\circ\psi_0)\circ(\varphi_0\circ \varphi)(a)-a\|\leq\|(\psi\circ\psi_0)\circ(\varphi_0\circ \varphi)(a)-\psi\circ\varphi(a)\|+\|\psi\circ\varphi(a)-a\|<\varepsilon$$ for all $a\in F$. This shows that $A$ is $p$-nuclear.
\end{proof}}

{

By a result of N. C. Phillips and M. G. Viola \cite[Theorem 6.2]{Lp AF}, every direct limit of $L^p$-operator algebras with $p$-completely isometric homomorphisms is still an $L^p$-operator algebra.

\begin{lem}\label{direct}
Assume that $p\in [1,\infty).$
\begin{enumerate}
  \item[(i)] Let $((A_i)_{i\in \mathbb{Z}_{>0}}, (\varphi_{j,i})_{i,j\in \mathbb{Z}_{>0}})$ be a direct system of $L^p$-operator algebras with $p$-completely isometric homomorphisms $\varphi_{j,i}: A_i\rightarrow A_j$ for $i,j\in\mathbb{Z}_{>0}$ with $i\leq j$. Then there exist an $L^p$-operator algebra $A$ and $p$-completely isometric maps $\psi_i: A_i\rightarrow A$ such that $A= \overline{\cup_{i=1}^\infty\psi_i(A_i)}$ and  $\psi_{j}\circ\varphi_{j,i}=\psi_i$ whenever $i\leq j$, that is, $A$ is a direct limit of $((A_i)_{i\in \mathbb{Z}_{>0}}, (\varphi_{j,i})_{i\in \mathbb{Z}_{>0}})$.
  \item[(ii)] Let $A$ be an $L^p$-operator algebra and $A_1\subset A_2\subset A_3\subset\cdots$ be
a sequence of norm closed subalgebras such that $A= \overline{\cup_{i=1}^\infty A_i }$. If for each $i\in \mathbb{Z}_{>0}$, $A_i$ is $p$-nuclear and there exists a $p$-completely contractive map $E_{i}: A\rightarrow A_{i}$ satisfying $E_{i}(a)=a$ for all $a\in A_{i}$, then $A$ is $p$-nuclear.
\end{enumerate}
\end{lem}

\begin{proof}
%Using a similar method in \cite[Proposition 6.2.4]{Rordam}, the direct limit $A$ can be identified as the closure of an union of increasing sequence of $L^p$-operator algebras.
%First, we construct a direct limit of $((A_i)_{i\in \mathbb{Z}_{>0}}, (\varphi_{j,i})_{i\leq j})$.

(i) We let $\Pi_{i=1}^\infty A_i$ and $\oplus_{i=1}^\infty
A_i$ denote the direct product  of $(A_i)_{i\in \mathbb{Z}_{>0}}$ and the the direct sum of $(A_i)_{i\in \mathbb{Z}_{>0}}$, respectively.
Let $\pi$ be the quotient map of $\Pi_{i=1}^\infty A_i$ onto $(\Pi_{i=1}^\infty A_i)/(\oplus_{i=1}^\infty
A_i)$.
For each $i$, define a homomorphism $\phi_i$ from
$A_i$ into $\Pi_{j=1}^\infty A_j$ as $a \longmapsto \left(\varphi_{j,i}(a)\right)_{j=1}^\infty$.
Denote $\psi_i=\pi\circ\phi_i$.

{\it Claim 1.} $\psi_{j}\circ\varphi_{j,i}=\psi_i$ whenever $i\leq j$.

Fix an $i\in \mathbb{Z}_{>0}$. Choose an element $a\in A_i$ and denote $c=\varphi_{i+1,i}(a)$. Then $c\in A_{i+1}$ and,
for $j\geq i+1$, we have
 \[ \varphi_{j,i}(a)= \varphi_{j,i+1}\circ \varphi_{i+1,i}(a)=\varphi_{j,i+1}(\varphi_{i+1,i}(a))=\varphi_{j,i+1}(c).\]
Since $\varphi_{s,t}=0$ whenever $s<t$, this shows that
$$\phi_i(a)- \phi_{i+1}(c)=(\varphi_{j,i}(a))_{j=1}^\infty-(\varphi_{j,i+1}(c))_{j=1}^\infty\in \bigoplus_{i=1}^\infty
A_i. $$
Then $\pi\circ \phi_i(a)=\pi\circ \phi_{i+1}(c)=\pi\circ \phi_{i+1}\circ\varphi_{i+1,i}(a)$,
that is, $\psi_i(a)=\psi_{i+1}\circ\varphi_{i+1,i}(a)$. Since $a\in A_i$ is arbitrary, we obtain
$\psi_i=\psi_{i+1}\circ\varphi_{i+1,i}.$
Using a recursive argument, one can see the claim. We note that $\psi_i(A_i)=\psi_{i+1}\circ\varphi_{i+1,i}(A_i)\subset \psi_{i+1}(A_{i+1})$
for all $i$.

{\it Claim 2.} For each $i\in \mathbb{Z}_{>0}$, $\psi_i$ is $p$-completely isometric.

%Since all $\varphi_{j,i}$'s are isometric, it follows that
%$$\|\psi_i(a)\|=\|\pi((\varphi_{k,i}(a))_{k=1}^\infty)\|=\limsup_{k\rightarrow \infty}\|\varphi_{k,i}(a)\|=\|a\|$$
%for all $i\in \mathbb{Z}_{>0}$ and $a\in A_i$.
%Hence each $\psi_i$ is isometric.

Fix $i,m\in\mathbb{Z}_{>0}$. We shall prove that $\mathrm{id}_{M_m^p}\otimes \psi_i$ is
isometric. For each $i\in\mathbb{Z}_{>0}$, we assume that $A_i\subset \mathcal{B}(L^p(X_i,\mu_i)$ is a norm  closed subalgebra. Then there exists an isometric representation of  $(\Pi_{i=1}^\infty A_i)/(\oplus_{i=1}^\infty
A_i)$ on
$E:=(\Pi_{i=1}^\infty L^p(X_i,\mu_i))/(\oplus_{i=1}^\infty L^p(X_i,\mu_i))$, which is isometrically isomorphic to an $L^p$-space $L^p(X,\mu)$ (see \cite[Theorem 3.3]{Heinrich}). Hence, for each $(a_{i})_{i=1}^\infty\in \Pi_{i=1}^\infty A_i$,
$\pi((a_{i})_{i=1}^\infty)$ can be viewed as an operator on $E$ given by
$$ \widetilde{\pi}((x_i)_{i=1}^\infty)\longmapsto \widetilde{\pi}((a_ix_i)_{i=1}^\infty),$$
where $\widetilde{\pi}$ is the quotient map from $\Pi_{i=1}^\infty L^p(X_i,\mu_i)$ onto $E$.

For $(a_{i})_{i=1}^\infty\in \Pi_{i=1}^\infty A_i$, one can see that \[ \|\pi((a_{i})_{i=1}^\infty)\|=\lim_{n\rightarrow\infty}\sup_{i\geq n}\|a_i\|=\limsup_{i\rightarrow\infty}\|a_i\|.\]
Likewise, given $m\in \mathbb{Z}_{>0}$ and
$(a_{i,s,t})_{i=1}^\infty\in \Pi_{i=1}^\infty A_i$ with $1\leq s,t\leq m$, we have
\[
\left\|\sum_{s,t=1}^m e_{s,t}\otimes \pi\Big((a_{i,s,t})_{i=1}^\infty\Big)\right\|
 = \limsup_{i\rightarrow \infty} \left\|\sum_{s,t=1}^me_{s,t}\otimes a_{i, s,t} \right\|.\]
 Then, for any $\{a_{s,t}\}_{1\leq s,t\leq m}\subset A_i$,
we have
$$\begin{aligned}
& \left\|\mathrm{id}_{M_m^p}\otimes \psi_i\left(\sum_{s,t=1}^m e_{s,t}\otimes a_{s,t}\right)\right\|
 = \left\|\sum_{s,t=1}^m e_{s,t}\otimes \big(\pi\circ \phi_i(a_{s,t})\big)\right\|
\\=& \left\|\sum_{s,t=1}^m e_{s,t}\otimes \pi\big( (\varphi_{k,i}(a_{s,t}))_{k=1}^\infty\big)\right\|
= \limsup_{k\rightarrow \infty}\left\|\sum_{s,t=1}^m e_{s,t}\otimes \varphi_{k,i}(a_{s,t})\right\|
\\=& \limsup_{k\rightarrow \infty}\left\|\sum_{s,t=1}^m e_{s,t}\otimes a_{s,t}\right\|=\left\|\sum_{s,t=1}^m e_{s,t}\otimes a_{s,t}\right\|.
\end{aligned}$$
Hence $\psi_i$ is $p$-completely isometric. This proves Claim 2.

Let $A=\overline{\cup_{i=1}^\infty\psi_i(A_i)}.$ Then, by Claim 1 and Claim 2, $(A, (\psi_i)_{i=1}^\infty)$ is a direct
limit of $((A_i)_{i\in \mathbb{Z}_{>0}}, (\varphi_{j,i})_{i\leq j})$.

(ii)  Arbitrarily choose a finite subset $F=\{x_1,\cdots,x_r\}$ of $A$ and
an $\varepsilon>0$. Since $A=\overline{\cup_{i=1}^\infty A_i}$, there exist a positive integer $N$ and $\{y_1,\cdots,y_r\}$ $\subset A_N$ such that $\max_{1\leq i\leq r}\|x_i-y_i\|<\frac{\varepsilon}{3}$. Since $A_{N}$ is $p$-nuclear, it follows that there exist a positive integer $k$ and two $p$-completely contractive maps $F:A_{N}\rightarrow M_k^p$ and $G: M_k^p\rightarrow A_{N}$ such that $\|G\circ F(y_i)-y_i\|<\frac{\varepsilon}{3}$ for all $i\in \{1,2,\cdots,r\}$.

By the hypothesis, there exists a $p$-completely contractive map $E_{N}:A\rightarrow A_{N}$ satisfying $E_{N}(y)=y$ for all $y\in A_{N}$. Let $\overline{F}=F\circ E_{N}$. Then $\overline{F}$ is a $p$-completely contractive map  from $A$ to $M_k^p$ and %$$\|\psi\circ\overline{F}(x_i)-x_i\|\leq\|\psi\circ\overline{F}(x_i)-\psi\circ\overline{F}(y_i)\|
%+\|\psi\circ\overline{F}(y_i)-y_i\|+\|y_i-x_i\|
%<\varepsilon$$
$$\begin{aligned}
\|G\circ\overline{F}(x_i)-x_i\|&\leq \|G\circ\overline{F}(x_i)-G\circ\overline{F}(y_i)\| + \|G\circ\overline{F}(y_i)-y_i\|+\|y_i-x_i\|\\
& <\frac{\varepsilon}{3}+\frac{\varepsilon}{3}+\frac{\varepsilon}{3}=\varepsilon \end{aligned}$$
 for $i\in \{1,2,\cdots,r\}$. Note that $G$ can be viewed as a $p$-completely contractive map
 from $M_k^p$ to $ A$. Hence $A$ is $p$-nuclear.
\end{proof}}

{

\begin{rmk} In Lemma \ref{direct} (ii), we do not know whether the condition on the existence of $(E_i)_{i\in \mathbb{Z}_{>0}}$
  can be cancelled. The existence of a $p$-completely contractive map $E_{i}: A\rightarrow A_{i}$ guarantees
  that each $p$-completely contractive map from $A_{i}$ to $M_k^p$ admits a $p$-completely contractive extension to $A$.
Let $V$ be a subspace of an operator space $W$ and let $H$ be a Hilbert space.
 The Arveson-Wittstock-Hahn-Banach theorem (see \cite[Theorem 4.1.5]{Ruan}) asserts that every
 completely contractive map $\varphi: V\rightarrow \mathcal{B}(H)$ can be extended to a completely contractive map on $W$. However, in the setting of $p$-operator spaces with $p\ne 2$, such an extension result does not hold in general (see \cite{Lee}).
\end{rmk}

%In the $C^*$-algebra setting, Arveson's Extension Theorem guarantees the existence of extensions of completely contractive maps and can be used to prove that direct limits of nuclear $C^*$-algebras are nuclear (see \cite [Exercise 2.3.7]{Brown and Ozawa}).

N. C. Phillips has pointed out that the spatial $L^p$-UHF algebras are $p$-nuclear (see \cite[Theorem 4.6]{Phillips look like}). Also this can be seen from Lemma \ref{direct}. }

{ Recall that $M_n^p=\mathcal{B}(l_n^p)$, $n=1,2,3, \cdots$.
The direct limit of $(M_n^p)_{n=1}^\infty$ is denoted by $\overline{M}_{\mathbb{Z}_{> 0}}^p$, where the inclusion map of $M_n^p$ into $M_{n+1}^p$ is given by
$$a   \longmapsto \begin{pmatrix}
a & 0 \\
0 & 0
\end{pmatrix}  .$$}
\begin{lem}\label{tensor}
Let $p\in [1,\infty)$ and $A$ be a norm closed subalgebra of {$\mathcal{B}(L^p(X,\mu))$}. Then
\begin{enumerate}
\item[(i)] given {any positive integer $n$}, $M_n^p\otimes_p A$ is $p$-nuclear if and only if $A$ is $p$-nuclear;
\item[(ii)] $\overline{M}_{\mathbb{Z}_{> 0}}^p\otimes_p A$  is $p$-nuclear if and only if $A$ is $p$-nuclear.
\end{enumerate}
\end{lem}

\begin{proof}
(i) ``$\Longleftarrow$".  Assume that $A$ is $p$-nuclear. {Arbitrarily choose a finite subset $F_1$} of $M_n^p\otimes_p A$ and an $\varepsilon>0$. Let $F_2$ be the collection of the matrix entries of elements in $F_1$. Clearly, $F_2$ is a finite subset of $A$. Since $A$ is $p$-nuclear, it follows that there exist a positive integer $k$ and two $p$-completely contractive maps $\varphi:A\rightarrow M_k^p$ and $\psi:M_k^p\rightarrow A$ such that $$\|\psi\circ\varphi(a)-a\|<\frac{\varepsilon}{n^2}$$ for all $a\in F_2$.
Then $$\left\|\big(\mathrm{id}_{M_n^p}\otimes \psi\big)\circ \big(\mathrm{id}_{M_n^p}\otimes \varphi\big)\left(\sum_{i,j}^ne_{i,j}\otimes a_{i,j}\right)-\left(\sum_{i,j}^ne_{i,j}\otimes a_{i,j}\right)\right\|<\varepsilon$$ for all $\sum_{i,j}^ne_{i,j}\otimes a_{i,j}\in F_1$. It follows that $M_n^p\otimes_p A$ is $p$-nuclear.

``$\Longrightarrow$". Assume that $M_n^p\otimes_p A$ is $p$-nuclear. Let $\iota:A\rightarrow M_n^p\otimes_p A$ be {the inclusion map} $a\mapsto e_{1,1}\otimes a$. Let $\rho:M_n^p\otimes_p A\rightarrow A$ be the linear map sending $\sum_{i,j}e_{i,j}\otimes a_{i,j}$ to $a_{1,1}$. One can check that $\iota$ is $p$-completely isometric  and $\rho$ is $p$-completely contractive.

{Arbitrarily choose a finite subset $F$ of $A$ and an $\varepsilon>0$.} Then $\iota(F)$ is a finite subset of $M_n^p\otimes_p A$. Since $M_n^p\otimes_p A$ is $p$-nuclear, it follows that there exist a positive integer $k$ and two $p$-completely contractive maps $\varphi:M_n^p\otimes_p A\rightarrow M_k^p$ and $\psi:M_k^p\rightarrow M_n^p\otimes_p A$ such that $$\|\psi\circ\varphi(T)-T\|<\varepsilon$$ for all $T\in \iota(F)$. Then we have $$\|(\rho\circ\psi)\circ(\varphi\circ\iota)(a)-a\|<\varepsilon$$ for all $a\in F$. Hence $A$ is $p$-nuclear.

(ii) The proof for the necessity is similar to that of (i) and hence is omitted.

``$\Longleftarrow$".
{Let $j$ be a positive integer and $P_j\in \mathcal{B}(l^p(\mathbb{Z}_{> 0}))$ be the projection given by
$(x_1,x_2,\cdots,x_{j},x_{j+1},x_{j+2},\cdots) \mapsto (x_1,x_2,\cdots,x_{j},0,0,\cdots)$.
Let $E=L^p(X,\mu)$ and $I_E$ be the identity operator on $E$. Then the map $$E_j:\overline{M}_{\mathbb{Z}_{> 0}}^p\otimes_p A\longrightarrow M_j^p\otimes_p A,\ \ \ \ (T\otimes a)\longmapsto (P_j\otimes I_E)(T\otimes a)(P_j\otimes I_E)$$ is $p$-completely contractive, $j=1,2,3,\cdots$.
Note that $\overline{M}_{\mathbb{Z}_{> 0}}^p\otimes_p A$ is the direct limit of $(M_j^p\otimes_p A)_{j\in \mathbb{Z}_{> 0}}$ with respect to the natural inclusion maps. By (i), each $M_j^p\otimes_p A$ is $p$-nuclear, it follows from Lemma \ref{direct} that $\overline{M}_{\mathbb{Z}_{>0}}^p\otimes_p A$ is $p$-nuclear.}
 \end{proof}

 %Let $F_3=\{a_1,a_2,\cdots,a_n\}$ be a finite subset of $\overline{M}_{\mathbb{Z}_{>0}}^p\otimes_p A$ and $\varepsilon>0$. Then there exist a positive integer $N$ and a finite subset $F_4=\{b_1,b_2,\cdots,b_n\}$ of $M_N^p\otimes_p A$ such that $$\|a_i-b_i\|<\frac{\varepsilon}{3}$$ for all $i\in\{1,2,\cdots,n\}$. Since $A$ is $p$-nuclear, by (i), it follows that $M_N^p\otimes_p A$ is $p$-nuclear. Then there exist a positive integer $k$ and two completely contractive maps $\sigma:M_N^p\otimes_p A\rightarrow M_k^p$ and $\delta:M_k^p\rightarrow M_N^p\otimes_p A$ such that $$\|\delta\circ\sigma(b_i)-b_i\|<\frac{\varepsilon}{3}$$ for all $i\in\{1,2,\cdots,n\}$. Let $P_N\in \mathcal{B}(l^p(\mathbb{Z}_{> 0}))$ be the projection given by
%$(x_1,x_2,x_3,\cdots) \mapsto (x_1,x_2,\cdots,x_{N},0,0,\cdots)$.
%Let $E=L^p(X,\mu)$ and $I_E$ be the identity operator on $E$. Then the map $$\kappa:\overline{M}_{\mathbb{Z}_{> 0}}^p\otimes_p A\rightarrow M_N^p\otimes_p A,\ \ \ \ (T\otimes a)\mapsto (P_N\otimes I_E)(T\otimes a)(P_N\otimes I_E)$$ is $p$-completely contractive. Notice that $\kappa(b_i)=b_i$ for $i\in \{1,2,\cdots,n\}$. Then we have
%$$\begin{aligned}
%\|\delta\circ\sigma\circ\kappa(a_i)-a_i\|&\leq \|\delta\circ\sigma\circ\kappa(a_i)-\delta\circ\sigma\circ\kappa(b_i)\|+ \|\delta\circ\sigma\circ\kappa(b_i)-b_i\|\\&+\|b_i-a_i\|
% <\frac{\varepsilon}{3}+\frac{\varepsilon}{3}+\frac{\varepsilon}{3}=\varepsilon. \end{aligned}$$
%Hence $\overline{M}_{\mathbb{Z}_{>0}}^p\otimes_p A$ is $p$-nuclear.

Let $p\in [1,\infty)$ and $E=L^p(X,\mu)$. Next we give a useful lemma concerning the left-regular representation of a discrete group $G$ on $l^p(G)\otimes_pE$.

We let $E^*$ denote the dual space of $E$.
For $x\in E$ and $y^*\in E^*$, we also write $\la x,y^*\ra$ for $y^*(x).$
For $T\in \mathcal{B}(E)$, we denote by $T^*$ its {\it adjoint} operator on $E^*$ defined by
$$T^*(\phi)(x):=\phi(Tx)$$ for $\phi\in E^*$ and $x\in E$. Clearly, $\la x,T^*(y^*)\ra=\la Tx,y^*\ra$ for $x\in E$ and {$y^*\in E^*$}.

\begin{lem}\label{adjoint}
Let $p\in [1,\infty)$, $E=L^p(X,\mu)$ and $A$ be a norm closed subalgebra of $\mathcal{B}(E)$.
Assume that $\iota_0: A\rightarrow \mathcal{B}(E)$ is the inclusion map with $(\iota,\lambda_p^E)$ being its associated regular covariant representation.
 If $q$ is the conjugate exponent of $p$ (that is, $1/p+1/q=1$), then
 $$[\lambda_p^E(s)]^*=\lambda_q^{E^*}(s^{-1})\ \ \textup{for all}  \ s\in G.$$
\end{lem}

\begin{proof}
Since $l^p(G)\otimes_p E$ can be identified as an $L^p$-direct sum of $E$, it follows from \cite[page 77, Exercise 4]{Conway} that the dual space of $l^p(G)\otimes_p E$ is isometrically isomorphic to $l^q(G)\otimes_q E^*$.
If $q=\infty$, then
$$l^q(G)\otimes_q E^*=\left\{(\eta_t)_{t\in G}: \sup_{t\in G}\|\eta_t\|<\infty,~\eta_t\in E^*\right\}.$$

For $t\in G$, we denote by $\delta_t$ the characteristic function of $\{t\}$.
Fix an $s\in G$. For $\sum_{t\in G}\delta_t\otimes \xi_t \in C_c(G,E)$ and $\sum_{r\in G}\delta_r\otimes \eta_r\in C_c(G,E^*)$,
we have $$\begin{aligned}
& \left\la \sum_{t\in G}\delta_t\otimes\xi_t,\ \lambda_p^E(s)^* \left(\sum_{r\in G}\delta_r\otimes\eta_r \right) \right\ra\\
=&\left\la \lambda^E_p(s)\left(\sum_{t\in G}\delta_t\otimes\xi_t \right),\sum_{r\in G}\delta_r\otimes\eta_r \right\ra\\
 =&\left\la\sum_{t\in G}\delta_{st}\otimes \xi_t,\ \sum_{r\in G}\delta_r\otimes\eta_r\right\ra
  = \sum_{t\in G}\big\la\xi_t,\eta_{st}\big\ra
  \\=&\left\la\sum_{t\in G}\delta_t\otimes\xi_t,\ \sum_{r\in G}\delta_{s^{-1}r}\otimes\eta_r\right\ra
  \\=&\left\la\sum_{t\in G}\delta_t\otimes\xi_t,\ \lambda_q^{E^*}(s^{-1})\left(\sum_{r\in G}\delta_r\otimes\eta_r\right)\right\ra .\end{aligned}$$
Since $C_c(G,E)$ is dense in $l^p(G)\otimes_p E$ and $C_c(G,E^*)$ is dense in $l^q(G)\otimes_q E^*$, it follows    that $\lambda_p^E(s)^*=\lambda_q^{E^*}(s^{-1})$.\end{proof}

\subsection{Proof of Theorem \ref{T:main}}

%by \cite[Definition 4.11]{N. C. Phillips Lp},
We first give a useful lemma.
\begin{lem}\label{exp}
Let $G$ be a countable discrete group with the unit $e$ and $E_e:F^p_\lambda(G,A,\alpha)\rightarrow A$ be the standard conditional expectation which maps $\sum_{s\in G} a_{s}\delta_s$ $\in$ $C_c(G,A,\alpha)$ to $a_e$. Then $E_e$ is $p$-completely contractive.
\end{lem}

\begin{proof}
To prove the lemma, we shall construct a linear map $\Phi$ from $F^p_\lambda(G,A,\alpha)$ to $\mathcal{B}(l^p(G)\otimes_p E)$ and
verify its $p$-complete contractivity.

{We denote by $\iota_0:A\rightarrow \mathcal{B}(E)$ the inclusion map with $(\iota,\lambda_p^E)$ being its associated regular covariant representation. By \cite[Lemma 3.19]{Phillips Lp Cuntz}, we may assume that $\iota\rtimes\lambda_p^E$ is an isometric representation of $F^p_\lambda(G,A,\alpha)$. Hence we identify $F^p_\lambda(G,A,\alpha)$ with the completion of $C_c(G,A,\alpha)$ under the operator norm of the representation $\iota\rtimes\lambda_p^E$. We also write $\iota\rtimes\lambda_p^E(f)$ for the typical element of $F^p_\lambda(G,A,\alpha)$ corresponding to $f=\sum_{t\in G}a_t\delta_t\in C_c(G,A,\alpha)$.}

Let $P_{e}$ be the rank-one projection of $l^p(G)$ to $l^p(\{e\})\subset l^p(G)$ and $I_E$ be the identity operator on $E$.
Identifying $l^p(G)\otimes_p E$ with the $L^p$-direct sum $\oplus_{t\in G}E$, we may simply take the $L^p$-direct sum representation
\[
\iota(a)=\oplus_{t\in G}\alpha_{t^{-1}}(a)\in \mathcal{B}\left(\oplus_{t\in G}E\right).\]
In the spatial $L^p$-operator tensor product picture, we have
\begin{align*}\iota(a)=\sum_{t\in G}\sigma_{t,t}\otimes \alpha_{t^{-1}}(a)
\end{align*}
where $\sigma_{t,t}\in \mathcal{B}(l^p(G))$ is defined by \[ \sigma_{t,t}(\delta_s) =\begin{cases}
\delta_t,& s=t,\\
0,& s\ne t,
\end{cases}\] and { the sum is defined as the limit of the net $\{\sum_{t\in F}\sigma_{t,t}\otimes \alpha_{t^{-1}}(a)\}_{F\in \mathcal{F}}$ in the strong operator topology, where $\mathcal{F}$ is the directed set consisting of all finite subsets of $G$ ordered by inclusion (see \cite[Definition 4.11]{Conway}).} Let $\lambda_{p}:G\rightarrow \mathcal{B}(l^p(G))$ be the left-regular representation.
For each $s\in G$, we have $\lambda_{p}^{E}(s) =\lambda_{p}(s)\otimes I_E$.
For any $f=\sum_{s\in G}a_s\delta_s\in C_c(G,A,\alpha)$, we have
\begin{align*}
&(P_e\otimes I_E)\big(\iota\rtimes\lambda_p^E(f)\big)(P_e\otimes I_E)\\
=& (P_e\otimes I_E)\left[\sum_{s\in G} \iota(a_s)\lambda_p^E(s)\right](P_e\otimes I_E)\\
=&(P_e\otimes I_E)\left[\sum_{s\in G}\left(\sum_{t\in G} \sigma_{t,t}\otimes\alpha_{t^{-1}}(a_s)\right)(\lambda_p(s)\otimes I_E)\right](P_e\otimes I_E)
  \\ =&\sum_{s\in G}(P_e\otimes I_E)\left[\sum_{t\in G}\sigma_{t,t}\otimes\alpha_{t^{-1}}(a_s)\right](\lambda_p(s)\otimes I_E)(P_e\otimes I_E)
  \\=&\sum_{s\in G}(P_e\otimes a_s)(\lambda_p(s)\otimes I_E)(P_e\otimes I_E)=P_e\otimes a_e.
\end{align*}

We define a linear map $\Phi: F^p_\lambda(G,A,\alpha)\rightarrow \mathcal{B}(l^p(G)\otimes_p E)$ by
$$\Phi(\iota\rtimes\lambda_p^E(f))=(P_e\otimes I_E)\left[\iota\rtimes \lambda_p^E(f)\right](P_e\otimes I_E)$$ for $f\in C_c(G,A,\alpha)$.
{For any positive integer $n$ and $\{f_{i,j}\}_{1\leq i,j\leq n}\subset C_c(G,A,\alpha)$, we have
 $$\begin{aligned}
&\left\|\mathrm{id}_{M_n^p}\otimes E_e\left(\sum_{i,j=1}^ne_{i,j}\otimes\left(\iota\rtimes\lambda_p^E(f_{i,j})\right)\right)\right\|=
\left\|\sum_{i,j=1}^ne_{i,j}\otimes f_{i,j}(e)\right\| \\
 =& \left\|P_e\otimes\sum_{i,j=1}^ne_{i,j}\otimes f_{i,j}(e)\right\|
  =\left\|\sum_{i,j=1}^n e_{i,j}\otimes (P_e\otimes f_{i,j}(e))\right\|
  \\=& \left\|\sum_{i,j=1}^ne_{i,j}\otimes \left[(P_e\otimes I_E)(\iota\rtimes\lambda_p^E(f_{i,j}))(P_e\otimes I_E)\right]\right\|
  \\= &\left\|\mathrm{id}_{M_n^p}\otimes \Phi\left(\sum_{i,j=1}^n e_{i,j}\otimes (\iota\rtimes\lambda_p^E)(f_{i,j})\right)\right\|. \end{aligned}$$
Hence,} in order to prove that $E_e$ is $p$-completely contractive, it suffices to show that $\Phi$ is $p$-completely contractive.
A direct computation shows that
$$\begin{aligned}
&\left\|\mathrm{id}_{M_n^p}\otimes \Phi\left(\sum_{i,j=1}^ne_{i,j}\otimes \left(\iota\rtimes\lambda_p^E(f_{i,j})\right)\right)\right\| \\
 =& \left\|\sum_{i,j=1}^ne_{i,j}\otimes \left[(P_e\otimes I_E)\left(\iota\rtimes\lambda_p^E(f_{i,j})\right)(P_e\otimes I_E)\right]\right\|
  \\ =&\left\|\Big(\mathrm{id}_{M_n^p}\otimes (P_e\otimes I_E)\Big)\left[ \sum_{i,j=1}^ne_{i,j}\otimes\left(\iota\rtimes_p^E(f_{i,j})\right)\right]\Big(\mathrm{id}_{M_n^p}\otimes (P_e\otimes I_E)\Big)\right\|
  \\ \leq &\left\|\sum_{i,j=1}^ne_{i,j}\otimes [\iota\rtimes\lambda_p^E(f_{i,j})]\right\|. \end{aligned}$$
Since $C_c(G,A,\alpha)$ is dense in $F^p_\lambda(G,A,\alpha)$, it follows that $\Phi$ is $p$-completely contractive.\end{proof}

Now we are going to give the proof of Theorem \ref{T:main}.

\begin{proof}[Proof of Theorem \ref{T:main}]
Since $G$ is amenable, it follows from \cite[Theorem 7.1]{Phillips look like} that $F^p(G,A,\alpha)$ is isometrically isomorphic to $F^p_\lambda(G,A,\alpha)$. Hence it suffices to prove that $F^p_\lambda(G,A,\alpha)$ is $p$-nuclear if and only if $A$ is $p$-nuclear. { As in the proof of Lemma \ref{exp}, we still assume that $\iota\rtimes\lambda_p^E$ is an isometric representation of $F^p_\lambda(G,A,\alpha)$, and
we write $\iota\rtimes\lambda_p^E(f)$ for the typical element of $F^p_\lambda(G,A,\alpha)$ corresponding to $f=\sum_{t\in G}a_t\delta_t\in C_c(G,A,\alpha)$.}

``$\Longrightarrow$". Assume that $F^p_\lambda(G,A,\alpha)$ is $p$-nuclear and $G$ is amenable with $e$ being its unit. We shall prove that $A$ is $p$-nuclear.

%
%Let $\rho$ be the canonical inclusion map from $A$ into $F^p_\lambda(G,A,\alpha)$, that is, $\rho(a)=\iota(a)\lambda_p^E(e)$, where $e$ is the unit of $G$. Since $l^p(\{1,2,\cdots,k\})\otimes_p l^p(G)\otimes_p E$ is isometrically isomorphic to $l^p(G)\otimes_p l^p(\{1,2,\cdots,k\})\otimes_p E$ and $\alpha$ is $p$-completely isometric,
%for positive integer $k$ and {$\{a_{i,j}\}_{1\leq i,j\leq k} \subset A$}, we have
%$$\begin{aligned}
% &\|(\mathrm{id}_{M_k^p}\otimes \rho)(\sum_{i,j=1}^ke_{i,j}\otimes a_{i,j})\|\\
% =&\|\sum_{i,j=1}^ke_{i,j}\otimes\iota(a_{i,j})\| = \|\sum_{i,j=1}^ke_{i,j}\otimes[\sum_{t\in G}\sigma_{t,t}\otimes\alpha_{t^{-1}}(a_{i,j})]\|
%  \\=&\|\sum_{t\in G}\sigma_{t,t}\otimes[\sum_{i,j=1}^k e_{i,j}\otimes\alpha_{t^{-1}}(a_{i,j})]\|
%   = \|\sum_{i,j=1}^ke_{i,j}\otimes a_{i,j}\|, \end{aligned}$$
%which shows that $\rho$ is $p$-completely isometric.

{Let $\rho$ be the canonical inclusion map from $A$ into $F^p_\lambda(G,A,\alpha)$, that is, $\rho(a)=\iota(a)\lambda_p^E(e)$, where $e$ is the unit of $G$. Clearly, $\rho$ is isometric. Since $\alpha$ is $p$-completely isometric, using the naturality of the construction of $\rho$, one can check that $\rho$ is $p$-completely isometric. }

Arbitrarily choose a finite subset $F$ of $A$ and an $\varepsilon>0$.
It is clear that $\rho(F)$ is a finite subset of $F^p_\lambda(G,A,\alpha)$. Since $F^p_\lambda(G,A,\alpha)$ is $p$-nuclear, there exist a positive integer $n$ and two $p$-completely contractive maps $\varphi:F^p_\lambda(G,A,\alpha)\rightarrow M_n^p$ and $\psi:M_n^p\rightarrow F^p_\lambda(G,A,\alpha)$ such that $$\|\psi\circ\varphi(f)-f\|<\varepsilon$$ for all $f\in \rho(F)$. Then we have $$\|(E_e\circ\psi)\circ(\varphi\circ\rho)(a)-a\|<\varepsilon$$ for all $a\in F$. Hence $A$ is $p$-nuclear.

``$\Longleftarrow$".
Arbitrarily fix a finite subset $F_1=\{x_1,x_2,\cdots,x_n\}$ of $F^p_\lambda(G,A,\alpha)$ and an $\varepsilon>0$. Then there exists a finite subset $F_2=\{f_1,f_2,\cdots,f_n\}$ of $C_c(G,A,\alpha)$ such that $$\|x_i-\iota\rtimes\lambda_p^E(f_i)\|<\frac{\varepsilon}{3}$$ for all $i\in\{1,2,\cdots,n\}$. Let $F_3$ be the union of the supports of all $f_i$'s. Then $F_3$ is a finite subset of $G$. Set $M=\max\{\|\iota\rtimes\lambda_p^E(f_i)\|: 1\leq i\leq n\}$. Since $G$ is amenable, it follows that there exists a finite subset $F$ of $G$ such that $$\max_{s\in F_3}\frac{|sF\bigtriangleup F|}{|F|}<\frac{\varepsilon}{3M},$$ where $|sF\bigtriangleup F|$ and $|F|$
denote the cardinalities of $sF\bigtriangleup F$ and $F$, respectively.

We let $M_F^p$ denote the collection of all $a\in \mathcal{B}(L^p(G))$ such that $a\xi=0$ whenever $\xi|_F=0$ and such that $a\xi\in l^p(F)\subset l^p(G)$ for all $\xi\in l^p(G)$. Moreover, we let $\{\sigma_{s,t}\}_{s,t\in F}$ denote the canonical matrix units of $M_F^p$.

Next we shall construct $p$-completely contractive maps $\varphi:F^p_\lambda(G,A,\alpha)\rightarrow M_{F}^p\otimes_p A$ and $\psi:M_{F}^p\otimes_p A\rightarrow F^p_\lambda(G,A,\alpha)$.

{\it Step 1.} The construction of $\varphi$.

Let $P$ be the finite-rank projection of $l^p(G)$ onto the linear span of $\{\delta_t:t\in F\}$.
Then, for $a\in A$,  we have $$(P\otimes I_E)\iota(a)=(P\otimes I_E) \iota(a)(P\otimes I_E)=\sum_{t\in F}\sigma_{t,t}\otimes \alpha_{t^{-1}}(a).$$
Thus one can see that $$\begin{aligned}
&(P\otimes I_E)\iota(a)\lambda_p^{E}(s)(P\otimes I_E)\\
=&(P\otimes I_E)\iota(a)\big(\lambda_p(s)\otimes I_E\big)(P\otimes I_E)
 \\=&
 \left(\sum_{t\in F}\sigma_{t,t}\otimes \alpha_{t^{-1}}(a)\right)\big(P\lambda_p(s)P\otimes I_E\big)
  \\=&\left(\sum_{t\in F}\sigma_{t,t}\otimes \alpha_{t^{-1}}(a)\right)
   \left(\sum_{r\in F\cap sF}\sigma_{r,s^{-1}r}\otimes I_E\right)
  \\=&\sum_{r\in F\cap sF}\sigma_{r,s^{-1}r} { \otimes}  \alpha_{r^{-1}}(a)\in M_{F}^p\otimes A. \end{aligned}$$
For $f=\sum_{s\in G} a_s\delta_s\in C_c(G,A,\alpha)$, we define $$\varphi(f)=(P\otimes I_E)\iota\rtimes \lambda_p^{E}(f)(P\otimes I_E).$$ Hence
$$\varphi(f)=\sum_{s\in G}\sum_{r\in F\cap sF}\sigma_{r,s^{-1}r}\otimes \alpha_{r^{-1}}(a_s)\in M_{F}^p\otimes A.$$
{Since $\iota\rtimes\lambda_p^E$ is an isometric representation of $F^p_\lambda(G,A,\alpha)$, it follows that
$$\|\varphi(f)\|=\|(P\otimes I_E)\iota\rtimes \lambda_p^{E}(f)(P\otimes I_E)\|\leq \|\iota\rtimes\lambda_p^E(f)\|.$$ Then $\varphi$ is contractive on $C_c(G,A,\alpha)$.}
So $\varphi$ can extend to a linear map from $F^p_\lambda(G,A,\alpha)$ to $M_{F}^p\otimes_p A$, still denoted by $\varphi$. For any positive integer $k$ and $\sum_{i,j=1}^k e_{i,j}\otimes f_{i,j}\in M_k^p\otimes_p F^p_\lambda(G,A,\alpha)$, we have
$$\begin{aligned}
&\left\|\mathrm{id}_{M_k^p}\otimes \varphi\left(\sum_{i,j=1}^k e_{i,j}\otimes f_{i,j}\right)\right\| =\left\|\sum_{i,j=1}^ke_{i,j}\otimes\varphi(f_{i,j})\right\|
 \\=&\left\|\sum_{i,j=1}^ke_{i,j}\otimes\left((P\otimes I_E)f_{i,j}(P\otimes I_E)\right)\right\|
  \\ =&\left\|\left(\mathrm{id}_{M_k^p}\otimes (P\otimes I_E)\right)\left( \sum_{i,j=1}^ke_{i,j}\otimes f_{i,j}\right)\left(\mathrm{id}_{M_k^p}\otimes (P\otimes I_E)\right)\right\|
  \\\leq &\left\|\sum_{i,j=1}^k e_{i,j}\otimes f_{i,j}\right\|,\end{aligned}$$
which implies that $\varphi$ is $p$-completely contractive.

{\it Step 2.} The construction of $\psi$.

We define $\psi:M_{F}^p\otimes_p A\rightarrow F^p_\lambda(G,A,\alpha)$ by $$\psi(\sigma_{s,t}\otimes a)=\frac{1}{|F|}\iota\left(\alpha_s(a)\right)\lambda_p^E({st^{-1}}).$$
Next we shall show that $\psi$ is $p$-completely contractive. Arbitrarily fix a positive integer $k$ and $T\in M_k^p\otimes_p (M_{F}^p\otimes_p A)$. It suffices to show that $$\left\|\mathrm{id}_{M_k^p}\otimes \psi(T)\right\|\leq\|T\|.$$

Assume that $T=\sum_{i,j=1}^ke_{i,j}\otimes \left(\sum_{s,t\in F}\sigma_{s,t}\otimes a_{s,t}^{i,j}\right)$, where $\{e_{i,j}\}_{i,j=1}^k$ is the canonical matrix units of $M_k^p$, $\{\sigma_{s,t}\}_{s,t\in F}$ is the canonical matrix units of $M_{F}^p$ and $a_{s,t}^{i,j}\in A$. Then we have $$\mathrm{id}_{M_k^p}\otimes \psi(T)=\frac{1}{|F|}\sum_{i,j=1}^ke_{i,j}\otimes\left[\sum_{s,t\in F}\iota\left(\alpha_s(a_{s,t}^{i,j})\right)\lambda_p^E(st^{-1})\right].$$

Arbitrarily choose $\xi\in l^p_k\otimes_p (l^p(G)\otimes_p E)$ and $\eta=(\eta_1,\eta_2,\cdots,\eta_k)$ $\in l^q_k\otimes_q (l^q(G)\otimes_q E^*)$, where $q$ is the conjugate exponent of $p$.
Identifying $l^p_k\otimes_p (l^p(G)\otimes_p E)$ with $k$ times $L^p$-direct sum of $l^p(G)\otimes_p E$, $\xi\in l^p_k\otimes_p (l^p(G)\otimes_p E)$ can be written as $\xi=(\xi_1,\xi_2,\cdots\xi_k)$, where $\xi_{i}\in l^p(G)\otimes_p E$ for  $i\in \{1,2,\cdots,k\}$.
Similarly, we may assume that $\eta=(\eta_1,\eta_2,\cdots,\eta_k)\in l^q_k\otimes_q (l^q(G)\otimes_q E^*)$, where $\eta_i\in l^q(G)\otimes_q E^*$ for $i\in \{1,2,\cdots,k\}$.

Denote $\widetilde{\xi}=(\widetilde{\xi}_1,\widetilde{\xi}_2,\cdots,\widetilde{\xi}_k)\in l^p_k\otimes_p l^p(F)\otimes_p (l^p(G)\otimes_p E)$, where $\widetilde{\xi}_i=(\widetilde{\xi}_{i,r})_{r\in F}$ and $\widetilde{\xi}_{i,r}= |F|^{-\frac{1}{p}} \lambda_p^E(r^{-1})\xi_i$.
Denote $\widetilde{\eta}=(\widetilde{\eta}_1,\widetilde{\eta}_2,\cdots,\widetilde{\eta}_k)\in l^q_k\otimes_q l^q(F)\otimes_q (l^q(G)\otimes_q E^*)$, where $\widetilde{\eta}_i=(\widetilde{\eta}_{i,r})_{r\in F}$ and $\widetilde{\eta}_{i,r}=|F|^{-\frac{1}{q}}\lambda_q^{E^*}(r^{-1})\eta_i$.
Then, by Lemma \ref{adjoint}, we have
$$\begin{aligned}
\left\la \mathrm{id}_{M_k^p}\otimes \psi(T)\xi,\ \eta \right\ra
&=\frac{1}{|F|}\sum_{s,t\in F}\left\la \begin{pmatrix}
 \sum_{j=1}^k\iota\left(\alpha_s(a_{s,t}^{1,j})\right)\lambda_p^E(st^{-1})\xi_j\\
 \sum_{j=1}^k\iota\left(\alpha_s(a_{s,t}^{2,j})\right)\lambda_p^E(st^{-1})\xi_j\\ \vdots\\ \sum_{j=1}^k\iota\left(\alpha_s(a_{s,t}^{k,j})\right)\lambda_p^E(st^{-1})\xi_j
\end{pmatrix},\ \begin{pmatrix}
 \eta_1\\
 \eta_2\\ \vdots\\ \eta_k
\end{pmatrix}\right\ra \\
&=\frac{1}{|F|}\sum_{s,t\in F}\sum_{i,j=1}^{k}\left\la \iota\left(\alpha_s(a_{s,t}^{i,j})\right)\lambda_p^E(st^{-1})\xi_j,\ \eta_i\right\ra
 \\&=\frac{1}{|F|}\sum_{i,j=1}^{k}\sum_{s,t\in F}\left\la \iota(a_{s,t}^{i,j})\lambda_{p}^E(t^{-1})\xi_j,\ \lambda_q^{E^*}(s^{-1})\eta_i\right\ra\\
&=\sum_{i,j=1}^k\left\la \sum_{s,t\in F}(\sigma_{s,t}\otimes \iota(a_{s,t}^{i,j})) \widetilde{\xi}_j,\ \widetilde{\eta}_i\right\ra\\
&=\left\la\left(\sum_{i,j}^k e_{i,j}\otimes\left(\sum_{s,t\in F}\sigma_{s,t}\otimes \iota\left(a_{s,t}^{i,j}\right)\right)\right)\widetilde{\xi},\ \widetilde{\eta}\right\ra.\end{aligned}$$
Note that if $\|\xi\|_p\leq 1$ and $\|\eta\|_q\leq 1$, then $\|\widetilde{\xi}\|_p\leq 1$ and $\|\widetilde{\eta}\|_q\leq 1$. This implies that $$\left\|\mathrm{id}_{M_k^p}\otimes \psi(T)\right\|\leq\left\|\sum_{i,j}^k e_{i,j}\otimes\left(\sum_{s,t\in F}\sigma_{s,t}\otimes \iota(a_{s,t}^{i,j})\right)\right\|.$$

%Since the action $\alpha$ is $p$-completely isometric, it follows that
%$$\begin{aligned}
%&\|\sum_{i,j}^k e_{i,j}\otimes\big(\sum_{s,t\in F}\sigma_{s,t}\otimes \iota(a_{s,t}^{i,j})\big)\|
%\\=&\|\sum_{i,j}^k e_{i,j}\otimes\big(\sum_{s,t\in F}\sigma_{s,t}\otimes (\sum_{r\in G}\sigma_{r,r}\otimes \alpha_{r^{-1}}(a_{s,t}^{i,j}))\big)\|
%\\=&\|\sum_{r\in G}\sigma_{r,r}\otimes\big(\sum_{i,j}^k e_{i,j}\otimes(\sum_{s,t\in F}\sigma_{s,t}\otimes \alpha_{r^{-1}}(a_{s,t}^{i,j}))\big)\|
%\\=&\|\sum_{i,j}^k e_{i,j}\otimes\big(\sum_{s,t\in F}\sigma_{s,t}\otimes \alpha_{r^{-1}}(a_{s,t}^{i,j})\big)\|
%\\=&\|\sum_{i,j}^k e_{i,j}\otimes(\sum_{s,t\in F}\sigma_{s,t}\otimes a_{s,t}^{i,j})\|=\|T\|.
%\end{aligned}$$

{ Since the action $\alpha$ is $p$-completely isometric, by the definition of $\iota$, it follows that $$\left\|\sum_{i,j}^k e_{i,j}\otimes\left(\sum_{s,t\in F}\sigma_{s,t}\otimes \iota(a_{s,t}^{i,j})\right)\right\|=\left\|\sum_{i,j}^k e_{i,j}\otimes\left(\sum_{s,t\in F}\sigma_{s,t}\otimes a_{s,t}^{i,j}\right)\right\|=\|T\|.$$}
Then we deduce that $\|\mathrm{id}_{M_k^p}\otimes \psi(T)\|\leq\|T\|$.
Hence $\psi$ is $p$-completely contractive.

{\it Step 3.} The verification of the $p$-nuclearity of $F^p_\lambda(G,A,\alpha)$.

Since $sF\Delta F=(sF\setminus(sF\cap F))\cup(F\setminus(sF\cap F))$, it follows that $|F\cap sF|=\frac{2|F|-|sF\bigtriangleup F|}{2}$.
For $i=1,2,\cdots,n$, we assume that $f_i=\sum_{t\in F_3}a^i_t\delta_t$.
Then we have
$$\begin{aligned}
&\|\psi\circ\varphi(x_i)-x_i\|\\
\leq&\|\psi\circ\varphi(x_i)-\psi\circ\varphi(f_i)\|+\|\psi\circ\varphi(f_i)-
\iota\rtimes\lambda_p^E(f_i)\|
 +\|\iota\rtimes\lambda_p^E(f_i)-x_i\|\\
\leq&\left\|\psi\left(\sum_{s\in F_3}\sum_{r\in F\cap sF}\sigma_{r,s^{-1}r}\otimes\alpha_{r^{-1}}(a_s^i)\right)-\sum_{s\in F_3}\iota(a_{s}^i)\lambda_p^E(s)\right\|+\frac{2\varepsilon}{3}
 \\=&
 \left\|\sum_{s\in F_3}\left[\sum_{r\in F\cap sF}\frac{1}{|F|}\iota(a_s^i)\lambda_p^E(s)-\iota(a_s^i)\lambda_p^E(s)\right]\right\|+\frac{2\varepsilon}{3}
  \\=&\left|\frac{|F\cap sF|}{|F|}-1\right|\cdot \left\|\sum_{s\in F_3}\iota(a_s^i)\lambda_p^E(s)\right\|+\frac{2\varepsilon}{3}
  \\<& \frac{\varepsilon}{3M}\cdot M+\frac{2\varepsilon}{3}=\varepsilon. \end{aligned}$$
Since $A$ is $p$-nuclear, it follows from Lemma \ref{tensor} that so is $M_{F}^p\otimes_p A$.
By Lemma \ref{matrix}, we conclude that $F^p_\lambda(G,A,\alpha)$ is $p$-nuclear.\end{proof}

Next we employ an example appeared in \cite[page 15]{PlpsOpenQues} to show that an amenable $L^2$-operator algebra may not be $2$-nuclear.

{\begin{eg}\label{Eg:amena}
Let $z=\mathrm{diag}(1,2)\in M_2$. We define a norm on $M_2$ by $\|a\|=\max\{\|a\|_2,\|zaz^{-1}\|_2\}$ for $a\in M_2$. Here $\|\cdot\|_2$ denotes the operator norm on $l^2_2$. This makes $M_2$ a Banach algebra. It is easy to check that the map $$\pi: M_2\rightarrow \mathcal{B}(l_4^2), \ \ \ a\mapsto\mathrm{diag}(a,zaz^{-1})$$ is an isometric representation of $M_2$. Hence $M_2$ is an $L^2$-operator algebra. Since $\pi(M_2)$ is isomorphic to $M_2^2$, it follows from \cite[Proposition 2.3.1]{Runde} that $M_2$ is amenable. Noting that $\pi(M_2)$ is not a selfadjoint subalgebra of $\mathcal{B}(l_4^2)$, it follows from \cite[Corollary 7.1.8]{Blecher and Le Merdy} that $\pi(M_2)$ is not $2$-nuclear.
\end{eg}}

\section{Proofs of corollaries}

In this section, we study $p$-nuclearity for $L^p$-Cuntz algebras and rotation $L^p$-operator algebras, and give the proofs
of Corollaries \ref{C:3} and \ref{2}.

We first recall an important notion which will be used in the proofs of Corollaries \ref{C:3} and \ref{2}.

\begin{defn}[{\cite[Definition 4.1]{Lp AF}}]\label{matrix norm}
Let $p\in [1,\infty)$ and $A$ be a separable $L^p$-operator algebra. We say that $A$ has
{\it unique $L^p$-operator matrix norms} if whenever $(X,\mathcal{B},\mu)$ and $(Y,\mathcal{C},\nu)$ are $\sigma$-finite measure spaces such that $L^p(X,\mu)$ and $L^p(Y,\nu)$ are separable, $\pi:A\rightarrow \mathcal{B}(L^p(X,\mu))$ and $\sigma:A\rightarrow \mathcal{B}(L^p(Y,\nu))$ are isometric representations, then $\sigma\circ\pi^{-1}:\pi(A)\rightarrow \sigma(A)$ is $p$-completely isometric.
\end{defn}

To prove  Corollary \ref{C:3}, we recall the definition of $L^p$-Cuntz algebras.

\begin{defn}[{\cite[Definition 1.1]{Phillips Lp Cuntz}}]
Let $d\in \{2,3,4,\cdots\}$. The {\it Leavitt algebra} $L_d$ is defined to be the universal complex associate algebra on generators $s_1,s_2,\cdots,s_d,t_1,t_2,\cdots,t_d$ satisfying $\sum_{j=1}^ns_jt_j=1 $ and
\[t_js_k =\begin{cases}
1,& j=k,\\
0,& j\ne k.
\end{cases}\]
\end{defn}

\begin{defn}[{\cite[Definition 8.8]{Phillips Lp Cuntz}}]\label{D:Cunz}
Let $d\in\{2,3,4,\cdots\}$ { and $p\in [1,\infty)$.} The $L^p$-Cuntz algebra $\mathcal{O}_d^p$ is defined to be the completion of $L_d$ in the norm $a\mapsto\|\rho(a)\|$ for any spatial representation $\rho$ of $L_d$ on some $L^p(X,\mu)$ for a $\sigma$-finite measure space $(X,\mathcal{B},\mu)$. Note that $\mathcal{O}_d^2$ is exactly the usual Cuntz algebra $\mathcal{O}_d$.
\end{defn}

Now we are going to show that the $L^p$-Cuntz algebra $\mathcal{O}_d^p$ is $p$-nuclear.

\begin{proof}[Proof of Corollary \ref{C:3}]
By \cite[Theorem 7.17]{N. C. Phillips Lp}, the stabilized $L^p$-Cuntz algebra $\overline{M}_{\mathbb{Z}\geq 0}^p\otimes_p \mathcal{O}_d^p$ is isometrically isomorphic to $F^p_\lambda(\mathbb{Z},B,\beta)$,
where $B$ is a stabilized  spatial $L^p$-UHF algebra and $\beta$ is an isometric action of $\mathbb{Z}$ on $B$ (see \cite[Notation 7.6]{N. C. Phillips Lp}). By \cite[Theorem 4.6]{Phillips look like}, the spatial $L^p$-UHF algebra is $p$-nuclear. Hence it follows from Lemma \ref{tensor} that $B$ is $p$-nuclear.

Let $\mathrm{id}_B$ be the identity representation of $B$, with $(\widetilde{\mathrm{id}_B},v)$ being its associated regular covariant representation. Since $B$ has unique $L^p$-operator matrix norms \cite[Corrollary 4.4 \& Proposition 6.4]{Lp AF}, it follows from \cite[Lemma 2.5]{WZ} that the integrated form $\widetilde{\mathrm{id}_B}\rtimes v$ is an isometric representation of $F^p_\lambda(\mathbb{Z},B,\beta)$.
From the definition of $\beta$ (see \cite[Notation 7.6]{N. C. Phillips Lp}), one can check that $\beta$ is $p$-completely isometric.
Then, by Theorem \ref{T:main}, $F^p_\lambda(\mathbb{Z},B,\beta)$ is $p$-nuclear and so is $\overline{M}_{\mathbb{Z}\geq 0}^p\otimes \mathcal{O}_d^p$. Using Lemma \ref{tensor} again, we deduce that $\mathcal{O}_d^p$ is $p$-nuclear.
\end{proof}

The rest of this section is devoted to showing that both $A^p_{\theta}$ and $F^p(\mathbb{Z},S^1,\alpha_{\theta})$ are $p$-nuclear. We first recall the definition of rotation $L^p$-operator algebras.

\begin{defn} [{\cite[Definition 1.1]{GarThiRotation14}}]
Let $\theta\in \mathbb{R}$. We define the {\it (algebraic) rotation algebra $R_\theta$} to be the universal associative complex algebra on generators $u$ and $v$ satisfying the relations \begin{enumerate}
\item[(i)] $u, v$ are both invertible, and
\item[(ii)] $uv=e^{2\pi \textrm{i}\theta} vu$.
 \end{enumerate}
\end{defn}

\begin{defn} [{\cite[Definition 1.5]{GarThiRotation14}}]
Let $\theta\in \mathbb{R}$ and let $E$ be a Banach space. An {\it isometric representation} of $R_\theta$ on $E$ is a representation $\rho:A\rightarrow \mathcal{B}(E)$ such that $\rho(u)$ and $\rho(v)$ are isometries of $E$.
\end{defn}

\begin{defn} [{\cite[Definition 1.9]{GarThiRotation14}}]\label{D:rota}
Let $\theta\in\mathbb{R}$ and $p\in[1,\infty)$. We denote by $\mathrm{Iso}_p$ the class of isometric representations of $R_\theta$ on $\sigma$-finite $L^p$-spaces.
Define a norm $\|\cdot\|_p$ on $R_\theta$ by
$$\|a\|_p=\sup\{\|\rho(a)\|:\rho\in \mathrm{Iso}_p\}.$$
The {\it rotation $L^p$-operator algebra $A_\theta^p$} is the completion of $R_\theta$ with respect to $\|\cdot\|_p$.
When $\theta$ is an irrational number, $A_\theta^p$ is called the {\it irrational rotation $L^p$-operator algebra}.
\end{defn}

Next we shall represent $A_\theta^p$ as a $L^p$-operator crossed product.
For $\theta\in \mathbb{R}$, we define an invertible isometry $W_\theta$ on $l^p(\mathbb{Z})$ as
\[ W_\theta\left(\sum_{j\in\mathbb{Z}}a_je_j\right)=\sum_{j\in\mathbb{Z}}a_je^{-2j\pi  \textup{i}\theta}e_j,\]
where $\{e_j\}_{j\in\mathbb{Z}}$ is the canonical basis of $l^p(\mathbb{Z})$.
We define $\beta_\theta(f)=W_\theta fW_\theta^{-1}$ for $f\in F^p_\lambda(\mathbb{Z})$.
Clearly, $\beta_\theta$ is a linear isometry.
Let $U$ be the bilateral shift operator on $l^p(\mathbb{Z})$. One can check that
$\beta_\theta(U)=e^{-2\pi \textup{i}\theta}U$. It follows that
$\beta_\theta\in \mathrm{Aut}(F^p_\lambda(\mathbb{Z}))$.
This induces an action of $\mathbb{Z}$ on $F_\lambda^p(\mathbb{Z})$ given by $n\longmapsto \beta_\theta^n=\beta_{n\theta}$, still denoted by $\beta_\theta$.
Since $F^p(\mathbb{Z})$ is isometrically isomorphic to $F^p_\lambda(\mathbb{Z})$, we identify $F^p(\mathbb{Z})$ with $F^p_\lambda(\mathbb{Z})$.
By \cite[Theorem 2.2]{GarThiRotation14}, $A_\theta^p$ is isometrically isomorphic to $F^{p}(\mathbb{Z},F^p(\mathbb{Z}),\beta_\theta)$.

Recall that $\alpha_\theta\in \textup{Aut}(C(S^1))$ is given by  $\alpha_{\theta}(f)(z)=f(e^{-2\pi \textrm{i}\theta} z)$ for $z\in S^1$.
One can check that $\alpha_\theta(\Gamma(f))=\Gamma(\beta_\theta(f))$, where $\Gamma: F^p(\mathbb{Z})\rightarrow C(S^1)$ is the Gelfand transform of $F^p(\mathbb{Z})$.
Clearly, $\Gamma$ induces a canonical injective homomorphism from $F^p(\mathbb{Z},F^p(\mathbb{Z}),\beta_\theta)$ to $ F^p(\mathbb{Z},S^1,\alpha_\theta)$, which is surjective precisely when $p=2$ (see \cite{GarThiRotation14}).

{\begin{lem}\label{ci}
Let $p\in[1,\infty)$ and $\theta\in\bR$. Then the action $\beta_\theta$ of $\mathbb{Z}$ on $F^p(\mathbb{Z})$ is $p$-completely isometric.
\end{lem}

\begin{proof}
Let $n$ be a positive integer.
If $\{f_{i,j}\}_{1\leq,i,j\leq n}\subset F^p(\mathbb{Z})$, then we have
$$\begin{aligned}
&\left\|\mathrm{id}_{M_n^p}\otimes \beta_\theta\left(\sum_{i,j=1}^ne_{i,j}\otimes f_{i,j}\right)\right\|
= \left\|\sum_{i,j=1}^ne_{i,j}\otimes W_\theta f_{i,j}W_\theta^{-1}\right\|
\\=&\left\|(I_{M_n^p}\otimes W_\theta)\left(\sum_{i,j=1}^n e_{i,j}\otimes f_{i,j}\right)(I_{M_n^p}\otimes W_\theta^{-1})\right\|
\\=&\left\|\sum_{i,j=1}^ne_{i,j}\otimes f_{i,j} \right\|,
\end{aligned}$$
where $I_{M_n^p}$ is the unit of $M_n^p$.
Hence $\beta_\theta$ is $p$-completely isometric.
\end{proof}}

To prove Corollary \ref{2}, {we need an extra lemma.}

Let $X$ be a compact metric space. We claim that $C(X)$ is an $L^p$-operator algebra for any $p\in[1,\infty)$. In fact, since $C(X)$ is separable (see \cite[page 136, Exercise 2]{Conway}) and, by \cite[3.7.2]{pedersen}, there is a positive linear functional on $C(X)$ which is separating for $C(X)$. It follows from the Riesz representation theorem that there exists a finite Borel measure $\mu$ on $X$ such that $\mu(U)>0$ for every nonempty open set $U\subset X$. Clearly, the map sending $f\in C(X)$ to the multiplication operator $M_f$ on $L^p(X,\mu)$ is an isometric representation of $C(X)$ on $L^p(X,\mu)$ for $p\in[1,\infty)$. This proves the claim.

%In the sequel we identify $f\in C(X)$ with $M_f\in \mathcal{B}(L^p(X,\mu))$.}

{ It is well known that $C(X)$ is 2-nuclear or, equivalently, $C(X)$ is nuclear  as a $C^*$-algebra (see \cite[Proposition 2.4.2]{Brown and Ozawa}).
Since the norms on $C(X)$ are the same when considered as $C^*$-algebras or $L^p$-operator
algebras, using a similar method, one can prove the following result.}

\begin{lem}\label{C(X)}
If $X$ is a compact metric space, then $C(X)$ is $p$-nuclear for all $p\in [1,\infty)$.
\end{lem}

Next we give the proof of Corollary \ref{2}.

\begin{proof}[Proof of Corollary \ref{2}]
First, we prove the $p$-nuclearity of $F^p(\mathbb{Z},S^1,\alpha_\theta)$.

Let $\mu$ be the normalized arc-length measure on $S^1$, and
let $\pi_0:C(S^1)\rightarrow \mathcal{B}(L^p(S^1,\mu))$ be the canonical isometric representation given by $f\mapsto M_f$ with $(\pi,v)$ being its associated regular covariant representation. By \cite[Proposition 4.6]{Lp AF}, $C(S^1)$ has unique $L^p$-operator matrix norms. Then the action $\alpha_\theta$ is $p$-completely isometric (see Remark \ref{completely isometric} (ii)), and the integrated form $\pi\rtimes v$ is an isometric representation of $F^p_\lambda(\mathbb{Z},S^1,\alpha_\theta)$ (see \cite[Lemma 2.5]{WZ}). By  Theorem \ref{T:main} and Lemma \ref{C(X)}, the reduced $L^p$-operator crossed product $F^p_\lambda(\mathbb{Z},S^1,\alpha_\theta)$ is $p$-nuclear. Since $\mathbb{Z}$ is amenable, it follows that $F^p(\mathbb{Z},S^1,\alpha_\theta)$ is also $p$-nuclear.

The rest is devoted to proving the $p$-nuclearity of $A^p_\theta$.
By the discussion after Definition \ref{D:rota}, we need only do this for $F^p(\mathbb{Z},F^p(\mathbb{Z}),\beta_\theta)$.

{  Let $\iota_0$ be the canonical isometric representation of $F^p(\mathbb{Z})$ on $l^p(\mathbb{Z})$ with $(\iota,w)$ being its associated regular covariant representation.

{\it Claim.} The integrated form $\iota\rtimes w$ is an isometric representation of the $L^p$-operator algebra $F^p_\lambda(\mathbb{Z},F^p(\mathbb{Z}),\beta_\theta)$.

We choose an isometric representation $\varphi_0$ of $F^p(\mathbb{Z})$ on a $\sigma$-finite $L^p$-space $E$ with $(\varphi, u)$ being its associated regular covariant representation.
Let $\mathrm{id}$ be the trivial action of $\mathbb{Z}$ on $M_n^p$. Since $\mathbb{Z}$ is amenable, it follows that $F^p(\mathbb{Z},M_n^p,\mathrm{id})$ is isometrically isomorphic to $F^p_\lambda(\mathbb{Z},M_n^p,\mathrm{id})$. Since $M_n^p$ has unique $L^p$ operator matrix norms and the action of $\mathbb{Z}$ on $M_n^p$ is trivial, by \cite[Lemma 2.6]{WZ}, it follows that $F^p_\lambda(\mathbb{Z},M_n^p,\mathrm{id})$ is isometrically isomorphic to $M_n^p\otimes_p F^p(\mathbb{Z})$.
Then we deduce that $F^p(\mathbb{Z},M_n^p,\mathrm{id})$ is isometrically isomorphic to $M_n^p\otimes_p F^p(\mathbb{Z})$.
By the universal property of $F^p(\mathbb{Z},M_n^p,\mathrm{id})$, we have $$\left\|\sum_{i,j=1}^n e_{i,j}\otimes \iota_0(a_{i,j})\right\|\geq \left\|\sum_{i,j=1}^n e_{i,j}\otimes \varphi_0(a_{i,j})\right\|$$ for all $[a_{i,j}]\in M_n^p\otimes_p F^p(\mathbb{Z})$. As in the proof of \cite[Lemma 2.6]{WZ}, one can show that $\|\iota\rtimes w(f)\|\geq \|\varphi\rtimes u(f)\|$ for $f\in C_c(\mathbb{Z},F^p(\mathbb{Z}),\beta_\theta)$. From the definition of reduced $L^p$-operator crossed products, one can see the claim.}

By Lemma \ref{ci}, the action $\beta_\theta$ is $p$-completely isometric. By \cite[Proposition 5.1 (a)]{an}, $F^p(\mathbb{Z})$ is $p$-nuclear.
Then, from Theorem \ref{T:main}, we deduce that the reduced $L^p$-operator crossed product $F^p_\lambda(\mathbb{Z},F^p(\mathbb{Z}),\beta_\theta)$ is $p$-nuclear. Since $\mathbb{Z}$ is amenable, it follows that $F^p(\mathbb{Z},F^p(\mathbb{Z}),\beta_\theta)$ is also $p$-nuclear. \end{proof}

%%%%%%%%%%%%%%%%%%%%%%%%%%%%%%%%%%%%%%%%%%%%%%%%%%%
\subsection*{Acknowledgements}
The authors wish to thank the referee for several helpful, constructive suggestions concerning the manuscript.

%
%
%%We wish to thank Professor E. Gardella for helpful comments and also for sharing their preprint on irrational rotation $L^p$-operator algebras.
%The first author is supported by National Natural Science Foundation of China (grant numbers 12201240, 12171195), China Postdoctoral Science Foundation (grant number 2022M711310), and Scientific Research Project of Putian University (grant number 2020001).
%The second author is supported by National Natural Science Foundation of China (grant number 12171195).

%\section*{Acknowledgements}

%The authors declare that all data supporting the findings of this study are available within the article.

%\subsection*{Data Availibility Statement}
%Data sharing not applicable to this article as no datasets were generated or analysed during the current study.

\subsection*{Funding}
This work was supported by the National Natural Science Foundation of China
(Grant numbers 12201240, 12171195).
%%%%%%%%%%%%%%%%%%%%%%%%%%%%%%%%%%%%%%%%%55

\end{document}